\documentclass[12pt,oneside,reqno]{amsart}
\usepackage[latin1]{inputenc}
\usepackage[english]{babel}
\usepackage{amsmath}
\usepackage{amssymb}
\usepackage{enumitem}
\usepackage[paper=a4paper]{geometry}
\usepackage{mathrsfs} 
\usepackage{mathtools}
\usepackage{hyperref}
\usepackage{color}
\usepackage[square,numbers]{natbib}

\title[IbPF for Bessel Bridges via Hypergeometric Functions]{Integration by Parts Formulae for  the Laws of  Bessel Bridges via Hypergeometric Functions}

\author{Henri Elad Altman}
\address{Imperial College London, UK}
\email{h.elad-altman@imperial.ac.uk}

\date{\today}

\newfont{\indic}{bbmss12}

\newcommand{\R}{\mathbb R}

\renewcommand{\d}{\, \mathrm{d}}

\theoremstyle{plain}
\newtheorem{thm}{Theorem}[section]
\newtheorem{lm}[thm]{Lemma}

\theoremstyle{definition}

\newtheorem{rk}[thm]{Remark}
\numberwithin{equation}{section}
\begin{document}


\begin{abstract}
In this article, we extend the integration by parts formulae (IbPF) for the laws of Bessel bridges recently obtained in \cite{EladAltman2019} to linear functionals. Our proof relies on properties of hypergeometric functions, thus providing a new interpretation of these formulae.
\end{abstract}

\maketitle

\section{Introduction}

\subsection{Bessel SPDEs}

Recently a family of stochastic PDEs which are infinite-dimensional analogues of Bessel processes were studied in \cite{EladAltman2019} and \cite{altman2019bessel}. These SPDEs define reversible dynamics for the laws of Bessel bridges, and have remarkable properties reminiscent of those of Bessel processes. In particular, they have the same scaling property as the additive stochastic heat equation, and are expected to arise as the scaling limits of several discrete dynamical interface models constrained by a wall. While the Bessel SPDEs of parameter $\delta \geq 3$, which are reversible dynamics for the  laws of Bessel bridges of dimension $\delta \geq 3$, had been introduced by Zambotti in the articles \cite{zambotti2002integration} and \cite{zambotti2003integration}, an open problem for several years was to extend the construction to $\delta <3$: apart from the derivation of an integration by parts formula for the special value $\delta=1$ - see\cite{zambotti2005integration} and
\cite{grothaus2016integration} - the extension to the whole regime $\delta<3$ had remained out of sight. This extension was a major challenge since, while the laws of Bessel bridges of dimension $\delta \geq 3$ can be represented as Gibbs measures with respect to the law of a Brownian bridge with an explicit, convex potential, such a representation fails for the laws of Bessel bridges of dimension $\delta < 3$, see Chap. 3.7 and 6.8 in \cite{zambotti2017random}. Indeed, the latter are not log-concave and, when $\delta < 2$, they are not even absolutely continuous with respect to the law of a Brownian bridge. In such a context, one in general cannot hope to construct an SPDE with the requested invariant measure. However, by exploiting the remarkable properties of Bessel bridges, the recent articles \cite{EladAltman2019} and \cite{altman2019bessel} have achieved this extension.

\subsection{Integration by parts formulae}
 
Let $C([0,1])$ be the space of continuous real-valued functions on $[0,1]$. By deriving integration by parts formulae (IbPF) for the laws of Bessel bridges of dimension $\delta< 3$ on the space $C([0,1])$, \cite{EladAltman2019} and \cite{altman2019bessel} have identified the structure that the corresponding SPDEs should have: namely, these SPDEs should contain a drift described by renormalised local times of the solutions (see (1.11)-(1.13) in \cite{EladAltman2019}), which is an analogue to higher orders of the principal value of local times appearing in the SDE satisfies by Bessel processes of dimension smaller than $1$, see e.g. Exercise 1.26 in \cite[Chap. XI.1]{revuz2013continuous}. The IbPF were also exploited to construct weak stationary solutions of these SPDEs in the special cases $\delta=1,2$, using Dirichlet form techniques, see  \cite[Section 5]{EladAltman2019} and \cite[Section 4]{altman2019bessel}. 

\subsection{Verification of the formulae for a different class of test functions}

The IbPF proved in Theorem 4.1 of \cite{EladAltman2019} and Theorem 3.1 of \cite{altman2019bessel} are valid for functionals of the form 
\begin{equation}
\label{exp_square_functional}
\Phi(X) = \exp \left( - \langle m, X^{2} \rangle \right), \qquad X \in C([0,1]),
\end{equation}
where we use the notation $\langle m, X^{2} \rangle = \int_0^1 X(r)^2 \, \d m(r)$, and where $m$ is any finite Borel measure on $[0,1]$. The reason for considering functionals as above is that \textit{squared} Bessel bridges possess a remarkable additivity property which allows to compute semi-explicitly their Laplace transform, see \cite[Chap. XI.3]{revuz2013continuous}. Note that observables defined by 
functionals of the form \eqref{exp_square_functional} characterize the laws of Bessel bridges, since those are supported on the set of \textit{non-negative} paths. It is nevertheless natural to ask whether the IbPF obtained in \cite{EladAltman2019} and \cite{altman2019bessel} still hold as such when one replaces functionals of the form \eqref{exp_square_functional} by more general ones.
In this article we show that these IbPF still hold for a very different class of test functionals. Namely, given a function $\varphi \in C([0,1])$, we consider the linear functional $\Phi$ defined on $L^2([0,1])$ by
\begin{equation}
\label{linear_functional}
\Phi(X) := \langle \varphi , X \rangle, 
\end{equation}
where we use the notation $\langle \varphi , X \rangle = \int_0^1 \varphi(r) X(r) \d r$. Note that, when $\varphi$ is not identically $0$, $\Phi$ is not bounded, and therefore may not be written as a function of the form \eqref{exp_square_functional}, so the results of \citep{EladAltman2019} and \cite{altman2019bessel} do not apply. However, it turns out that the IbPF still hold for such a functional $\Phi$. 
One striking feature of these formulae is the fact that, when $\delta <3$, they involve a renormalisation procedure using Taylor polynomials either of order $0$ (for $\delta \in (1,3)$) or of order $2$ (for $\delta \in (0,1)$), however there is no regime where only first-order renormalisation is required, as one would expect in the window $\delta \in (1,2)$. This absence of transition at $\delta=2$ was already observed in \cite[Remark 4.3]{EladAltman2019} for functionals $\Phi$ of the form \eqref{exp_square_functional}. Note that those functionals are very special, in particular they depend smoothly in $X^2$. On the other hand, non-zero functionals of the form \eqref{linear_functional}
depend smoothly on $X$ but not on $X^2$,
 however the absence of transition at $\delta=2$ holds for such functionals as well. Moreover, in a forthcoming article, we will show that a similar phenomenon actually holds for any functional $\Phi: L^2(0,1) \to \R$ which is bounded, $C^1$, with bounded Fr\'{e}chet differential. All these results support the conjecture, raised in \cite{EladAltman2019}, that the  first-order derivative of the diffusion local times of the solutions to the Bessel SPDEs must vanish at $0$, so that the drift term
appearing in these SPDEs needs to be renormalised at order 0 and 2, for $\delta \in (1,3)$ and $\delta \in (0,1)$ respectively, but never at order 1: see Remark \ref{sigma_fctn_b_square} below.

\subsection{Hypergeometric functions}

The proof of the IbPF for functionals of the form \eqref{linear_functional} has its own interest, as it provides an interpretation of the IbPF using properties of hypergeometric functions. More precisely, we exploit the fact that two-point functions of Bessel bridges can be written using hypergeometric functions, see \eqref{twopt} below. This fact is reminiscent of Cardy's formula for Bessel processes which, for the special value $\delta=5/3$, admits an interpretation in terms of the crossing probability for a critical percolation model: see \cite[Chap. 1.3]{katori2016bessel}.    

\section{The formulae for linear functionals}

Henceforth, as in \cite{EladAltman2019}, for all $\delta>0$, we denote by $P^\delta$ the law, on $C([0,1])$, of a $\delta$-dimensional Bessel bridge from $0$ to $0$ on $[0,1]$, and let $E^\delta$ denote the associated expectation operator (see \cite[Chap XI.3]{revuz2013continuous} for the definition of Bessel bridges). For all $b \geq 0$ and $r\in(0,1)$, we set as in Def. 3.4 of \cite{EladAltman2019}
\begin{equation}\label{Sigma}
\Sigma^\delta_r({\rm d}X \,|\, b) := \frac{p^{\delta}_{r}(b)}{b^{\delta-1}} \,
 P^{\delta} [ {\rm d} X \,| \, X_{r} = b],
\end{equation}
where $P^{\delta} [ {\rm d} X \,| \, X_{r} = b]$ is the law of a $\delta$-Bessel bridge between $0$ and $0$ pinned at $b$ at time $r$, see \cite[Section 3.3]{EladAltman2019}, and $p^{\delta}_{r}$ is the probability density function of $X_{r}$ under $P^{\delta}$,  given by
\[ p^{\delta}_{r}(b) =
\frac{b^{\delta-1}}{2^{\frac\delta2-1}\,\Gamma(\frac{\delta}{2})(r(1-r))^{\delta/2}}\, \exp \left(- \frac{b^{2}}{2r(1-r)} \right), \quad b \geq 0 . \]
We also recall the definition of a family of Schwartz distributions on $[0,\infty)$, denoted by $(\mu_{\alpha})_{\alpha \in \R}$, that plays an important role in the IbPF: 

\begin{itemize}
\item if $ \alpha = -k$ with $k \in \mathbb{N} \cup \{0\}$, we set 
\[
\langle \mu_{\alpha}, \psi \rangle := (-1)^{k} \psi^{(k)}(0), \qquad \forall \,
\psi \in S([0,\infty))
\]
\item else, we set
\[
\langle \mu_{\alpha} , \psi \rangle := \int_{0}^{+ \infty} \left(
\psi(x) - \sum_{0\leq j\leq 
-\alpha} \frac{x^{j}}{j!} \, \psi^{(j)}(0) \right) \frac{x^{\alpha -1}}{\Gamma(\alpha)} \d x, \quad \forall \,
\psi \in S([0,\infty)),
\]
\end{itemize}
where $S([0,\infty))$ is the family of $C^\infty$ functions $\psi: [0,\infty) \to \mathbb{R}$ such that, for all $k, l \geq 0$, there exists $C_{k,\ell} \geq0$ satisfying
\[
 | \psi^{(k)} (x)  | \, x^{\ell} \leq C_{k,\ell}, \qquad \forall x \geq 0.
\]
In addition, for any Fr\'{e}chet differentiable $\Phi: L^2([0,1]) \to \mathbb{R}$ and any $h \in L^2([0,1])$, we denote by $\partial_h \Phi$ the directional derivative of $\Phi$ along $h$:
\[ \partial_h \Phi(X) = \underset{\epsilon \to 0}{\lim} \, \frac{\Phi(X+\epsilon h) - \Phi(X)}{\epsilon}, \qquad X \in L^2([0,1]).\]
In particular, for $\Phi$ of the form \eqref{linear_functional}, $\partial_h \Phi(X) = \langle \varphi, h \rangle$ for all $X \in L^2([0,1])$. Finally, we denote by $C^{2}_{c}(0,1)$ the space of $C^{2}$ functions compactly supported in $(0,1)$. With these notations at hand, we may now state the main result of this article.

\begin{thm}
\label{ibpf_linear_test_functn}
Let $\delta >0$. For all $\varphi \in C([0,1])$, setting $\Phi(X)=\langle \varphi, X \rangle$, then for all $h \in C^{2}_{c}(0,1)$ we have 
\begin{equation}
\label{ibpf}
\begin{split}
E^{\delta} (\partial_{h} \Phi (X) ) &= - E^{\delta} [\langle h '' , X \rangle \Phi(X)] \\
& - \frac{\Gamma(\delta)}{4(\delta-2)} \int_{0}^{1} dr \, h(r) \, \langle \mu_{\delta-3} (d b) , \Sigma^{\delta}_{r} (\Phi | b) \rangle.
\end{split}
\end{equation}
\end{thm}

\begin{rk}
\label{sigma_fctn_b_square}
By Lemma \ref{expr_sigma_power_series} below, for a functional $\Phi$ of the form \eqref{linear_functional}, and for all $r \in (0,1)$, $\Sigma^{\delta}_{r} (\Phi | b)$ is a smooth function of $b^2$, so in particular
\begin{equation}
\label{vanishing_derivative_sigma}
\frac{d}{d b} \Sigma^{\delta}_{r} (\Phi | b) \big \rvert_{b=0} = 0.
\end{equation}
Recalling the definition of the distribution $\mu_{\delta}$, we thus retrieve from \eqref{ibpf} the formulae of Theorem 4.1 in \cite{EladAltman2019}. Note in particular that, due to \eqref{vanishing_derivative_sigma}, the apparent singularity at $\delta=2$ due to the term $\frac{1}{\delta-2}$ is cured by the vanishing at $\delta=2$ of $\langle \mu_{\delta-3} ({\rm{d}} b) , \Sigma^{\delta}_{r} (\Phi | b) \rangle$. 
The vanishing property \eqref{vanishing_derivative_sigma} was already observed in \cite{EladAltman2019} and \cite{altman2019bessel} when $\Phi$ is of the form \eqref{exp_square_functional}: for such functionals, which are very special as they depend smoothly on $X^2$, it was noted that $\Sigma^{\delta}_{r} (\Phi | b)$ is a smooth function of $b^2$, but it was unclear whether $\Sigma^{\delta}_{r} (\Phi | b)$ has a more complicated dependence on $b$ for more general functionals $\Phi$. On the other hand \eqref{vanishing_derivative_sigma} above shows that the smoothness of $\Sigma^{\delta}_{r} (\Phi | b)$ in $b^2$ remains true even when $\Phi(X)$ is not smooth in $X^2$, as is the case for non-zero functionals $\Phi$ of the form \eqref{linear_functional}.
From the dynamical viewpoint, this supports the conjecture, proposed in \cite{EladAltman2019} and \cite{altman2019bessel} that, for all $x \in (0,1)$, the family of diffusion local times $(\ell^b_{t,x})_{b,t \geq 0}$ of the process $(u(t,x))_{t \geq 0}$, where $u$ is a solution to the Bessel SPDE of parameter $\delta$, satisfies
\[ \frac{\partial}{\partial b} \ell^b_{t,x} \big \rvert_{b=0} = 0. \]
As a consequence, the Taylor polynomials based at $b=0$ of $\ell^b_{t,x}$ are even.  
Thus, the Taylor remainders 
appearing in the Bessel SPDEs (1.11)-(1.13) in \cite{EladAltman2019} jump, as $\delta$ goes below $1$, from 0th order to 2nd order, and there is no window for $\delta$ where the SPDE involves a renormalisation of purely order $1$.  
\end{rk}

\begin{rk}
While \cite{EladAltman2019} proved IbPF for the laws of Bessel bridges from $0$ to $0$, \cite{altman2019bessel} extended these formulae to the case of bridges with arbitrary endpoints $a,a' \geq 0$. In this article, we are considering for simplicity the former case, for which the interpretation in terms of hypergeometric functions is more transparent, but we believe Theorem \ref{ibpf_linear_test_functn} remains true for bridges with arbitrary endpoints as well.
\end{rk}
In the remainder of this article, we prove Theorem \ref{ibpf_linear_test_functn}. Note that given the linearity of our test functional $\Phi =\langle \varphi, \cdot \rangle$, the above formula can be rewritten in the following way:
\begin{equation}
\label{equality_zeta}
\begin{split}
\langle \varphi, h \rangle &= - \int_{0}^{1} \varphi(s) \int_{0}^{1} h''(r) E^{\delta}\left[X_{s} X_{r} \right]  \d r \d s\\
& - \frac{\Gamma(\delta)}{4(\delta-2)} \int_{0}^{1} \d s \, \varphi(s) \int_{0}^{1} \d r \, h(r) \, \langle \mu_{\delta-3} ({\rm{d}} b) , \Sigma^{\delta}_{r} (X_s | b) \rangle.
\end{split}
\end{equation}
In the last line, we used that, for all $r \in (0,1)$
\begin{equation}
\label{interversion}
\langle \mu_{\delta-3} (d b) , \Sigma^{\delta}_{r} (\Phi(X) | b) \rangle = \int_0^1 \d s \, \varphi(s) \langle \mu_{\delta-3}, \Sigma^{\delta}_{r} (X_s | b) \rangle. 
\end{equation}
We will first justify this interversion. To do so we invoke the following result which shows that, for all $r \in (0,1)$, the function $(s,b) \to \Sigma^{\delta}_{r} (X_s | b)$ is analytic on the domain $(s,b) \in (0,1) \setminus \{r\} \times \mathbb{R}_+$:
\begin{lm}
\label{expr_sigma_power_series}
For all $r,s \in (0,1)$, $r \neq s$, and $b \geq 0$, we have
\begin{equation*}
\Sigma^{\delta}_{r} (X_s | b) = \frac{1}{2^{\delta/2-1}(r(1-r))^{\delta/2}} \exp \left(- \frac{D(s,r)}{2} \, b^2 \right) \sum_{k=0}^{\infty} C_k f_k(s,r) \, b^{2k}, 
\end{equation*}
where
\[D(s,r) :=  \mathbf{1}_{\{s < r\}} \frac{1-s}{(r-s)(1-r)} + \mathbf{1}_{\{s > r\}} \frac{s}{r(s-r)}, \]
and, for all $k \geq 0$
\[C_{k} := \frac{\Gamma(k + \frac{\delta+1}{2})}{\Gamma(\delta/2) \, \Gamma(k+\delta/2) \, k!}, \]
and
\[ f_{k}(s,r) = \frac{\mathbf{1}_{\{s < r\}}}{\left(2 \left(r - s\right) \right)^{k-\frac{1}{2}}} \left( \frac{s}{r} \right)^{k+1/2}
+  \frac{\mathbf{1}_{\{s > r\}}}{\left(2 \left(s - r\right) \right)^{k-\frac{1}{2}}} \left( \frac{1 - s}{1 -r} \right)^{k+1/2}. \]
\end{lm}  

\begin{proof}
Assume for instance that $s < r$. Then, the joint law of $(X_s,X_r)$ on $[0,\infty)^2$, when $X$ is distributed as $P^\delta$, is given in terms of the transition densities $(p^\delta_t(x,y))_{t>0, x,y \geq 0}$ of a $\delta$-dimensional Bessel process by
\begin{equation} 
\label{joint_density_bessel}
p^\delta_s(0,a) \, p^\delta_{r-s}(a,b) \, \frac{p^\delta_{1-r}(b,0)}{p^\delta_{1}(0,0)} \, d a \, d b, 
\end{equation}
where we use the notation
\[ \frac{p^\delta_{1-r}(b,0)}{p^\delta_{1}(0,0)} = \lim_{\epsilon \to 0} \frac{p^\delta_{1-r}(b,\epsilon)}{p^\delta_{1}(0,\epsilon)}, \]
see \cite[Chap XI.3]{revuz2013continuous}. Therefore, for all $b \geq 0$,
\[ \Sigma^{\delta}_{r} (X_s | b) = \frac{p^{\delta}_{r}(b)}{b^{\delta-1}} E^{\delta}_{r}[X_s | X_r = b] = \int_0^\infty \frac{p^\delta_s(0,a) p^\delta_{r-s}(a,b)}{b^{\delta-1}} \frac{p^\delta_{1-r}(b,0)}{p^\delta_{1}(0,0)} a \, d a .\]
Recalling from \cite[Chap. XI.1]{revuz2013continuous} that, for all $a,b > 0$,
\[p^\delta_s(0,a) = \frac{a^{\delta-1}}{2^{\delta/2-1} \, s^{\delta/2} \, \Gamma(\delta/2)} \exp \left(- \frac{a^2}{2s} \right), \]
\[p^\delta_{r-s}(a,b) = \frac{b}{r-s} \left(\frac{b}{a}\right)^{\delta/2-1} \exp \left(- \frac{a^2+b^2}{2(r-s)} \right) \sum_{k=0}^\infty \frac{\left(\frac{ab}{2(r-s)}\right)^{2k + \delta/2-1}}{k! \, \Gamma(k+\delta/2)}, \] 
\[\frac{p^\delta_{1-r}(b,0)}{p^\delta_{1}(0,0)} = (1-r)^{-\delta/2} \exp \left(- \frac{b^2}{2(1-r)} \right), \]
the result follows at once by applying Fubini and by computations of integrals in terms of the $\Gamma$ function. 
\end{proof}
As a consequence, we deduce that the equality \eqref{interversion} holds for all $r \in (0,1)$.
Indeed, since $\mu_{\delta-3}$ is the distributional third-order derivative of  $\mu_{\delta}$ (see Prop 2.5 in \cite{EladAltman2019}), we have
\[\begin{split}
\langle \mu_{\delta-3} (d b) , \Sigma^{\delta}_{r} (\Phi(X) | b) \rangle &= - \langle \mu_{\delta} ({\rm{d}} b) , \frac{\d^3}{\d b^3} \Sigma^{\delta}_{r} (\Phi(X) | b) \rangle \\
&= - \frac{1}{\Gamma(\delta)} \int_0^\infty \d b \, b^{\delta-1} \frac{\d^3}{\d b^3} \Sigma^{\delta}_{r} (\Phi(X) | b),
\end{split}\]
and Lemma \ref{expr_sigma_power_series} ensures that
\begin{equation}
\label{fubini}
\int_0^1 \, d s \int_0^\infty \d b \, b^{\delta-1} \left| \frac{\d^3}{\d b^3} \Sigma^{\delta}_{r} (X_s | b) \right| < \infty. 
\end{equation}
Hence, we deduce that
\[\begin{split}
\langle \mu_{\delta-3} (d b) , \Sigma^{\delta}_{r} (\Phi(X) | b) \rangle &= - \langle \mu_{\delta} (d b) , \frac{\d^3}{\d b^3} \Sigma^{\delta}_{r} (\Phi(X) | b) \rangle \\
&= - \int_0^1 \d s \, \varphi(s) \, \frac{1}{\Gamma(\delta)} \int_0^\infty \d b \, b^{\delta-1} \frac{\d^3}{\d b^3} \Sigma^{\delta}_{r} (X_s | b) \\
&= - \int_0^1 \d s \, \varphi(s) \, \langle \mu_{\delta}, \frac{\d^3}{\d b^3} \Sigma^{\delta}_{r} (X_s | b) \rangle \\
&= \int_0^1 \d s \, \varphi(s) \, \langle \mu_{\delta-3}, \Sigma^{\delta}_{r} (X_s | b) \rangle,
\end{split}\]
where an application of Fubini justified by \eqref{fubini} was used to obtain the second line. Hence, the claimed equality \eqref{interversion} follows, and the proof of Theorem \ref{ibpf_linear_test_functn} indeed reduces to establishing the equality \eqref{equality_zeta}. To prove the latter, it suffices to prove that the following equality holds $ds$-almost-everywhere:
\begin{align*}
 h(s) &= - \int_{0}^{1} h''(r) E^{\delta}\left[X_{s} X_{r} \right]  \d r\\
& - \frac{\Gamma(\delta)}{4(\delta-2)} \int_{0}^{1} \d r \, h(r) \, \langle \mu_{\delta-3} ({\rm{d}} b) , \Sigma^{\delta}_{r} (X_s | b) \rangle.
\end{align*}
In turn, the latter equality will follow upon showing that, for all $s \in (0,1)$, the function $r \mapsto E^{\delta}\left[X_{r} X_{s}\right]$ satisfies the following equality of distributions on $(0,1)$:
\begin{equation}
\label{eq_distrib}
\begin{split}
 \frac{d^{2}}{dr^{2}} E^{\delta} \left[X_{r} X_{s}\right] &= - \delta_{s}(r)\\
& - \frac{\Gamma(\delta)}{4(\delta-2)} \langle \mu_{\delta-3} ({\rm{d}} b) , \Sigma^{\delta}_{r} (X_s | b) \rangle,
\end{split} 
\end{equation}
where $\delta_s$ denotes the Dirac measure at $s$. The proof of \eqref{eq_distrib} will rely on the explicit computation of second moments of Bessel bridges using hypergeometric functions. 
\begin{proof}[Proof of equality \eqref{eq_distrib}]
\textbf{First step:} We start by showing that, for all $s \in (0,1)$, the function $r \mapsto E^{\delta}\left[X_{r} X_{s}\right]$ is twice differentiable for $r \in (0,1)\setminus\{s\}$, and that
\begin{equation}
\label{scd_der}
\frac{d^{2}}{dr^{2}} E^{\delta} \left[X_{r} X_{s}\right] = - \frac{\Gamma(\delta)}{4(\delta-2)} \langle \mu_{\delta-3} ({\rm{d}} b) , \Sigma^{\delta}_{r} (X_s | b) \rangle.
\end{equation}
Assume for instance that $0<s<r<1$. Then, using the expression \eqref{joint_density_bessel} for the joint density of $(X_s,X_r)$, where $X \overset{(d)}{=} P^\delta$, we obtain
\begin{equation}
\label{twopt}
E^{\delta}[X_{s}X_{r}] = 2 \frac{\Gamma \left( \frac{\delta+1}{2} \right)^{2}}{ \Gamma \left( \frac{\delta}{2} \right)^{2}} \frac{(r-s)^{\delta/2+1} \left( s (1-r) \right)^{1/2}}{\left( r (1-s) \right)^{\frac{\delta+1}{2}}} \, {}_{2}F_{1} \left( \frac{\delta+1}{2}, \frac{\delta+1}{2}, \frac{\delta}{2}, \frac{s(1-r)}{r(1-s)} \right),
\end{equation}
while, by Lemma \ref{expr_sigma_power_series}, the right-hand side of \eqref{scd_der} equals
\begin{equation}
\label{rhs}
- \frac{1}{2} \frac{\Gamma \left( \frac{\delta+1}{2} \right)^{2}}{ \Gamma \left( \frac{\delta}{2} \right)^{2}} \frac{(r-s)^{\delta/2-1} s^{1/2}}{ (1-r)^{3/2} r^{\frac{\delta+1}{2}} (1-s)^{\frac{\delta-3}{2}}} \ {}_{2}F_{1} \left( \frac{\delta+1}{2}, \frac{\delta-3}{2}, \frac{\delta}{2}, \frac{s(1-r)}{r(1-s)} \right)
\end{equation}
where ${}_{2}F_{1}$ denotes the hypergeometric function. Recall that  the hypergeometric function ${}_{2}F_{1}$ is defined, for all $a,b,c \in \mathbb{C} \setminus \mathbb{Z}_{-}$, and all $z \in \mathbb{C}$ such that $|z| < 1$, by
\[ {}_{2}F_{1}(a,b,c,z) := \sum_{k=0}^{+\infty} \frac{(a)_{k} (b)_{k}}{k! (c)_{k}} z^{k} \]
where, for any $\alpha > 0$ and $k \geq 0$, $(\alpha)_{k}:= \begin{cases} \ 1, \ &\text{if} \ k = 0 \\  \alpha (\alpha +1) \ldots (\alpha + k -1), \ &\text{if} \ k \geq 1\end{cases} $. 
\\
Note that the second argument of the hypergeometric function appearing in $\eqref{twopt}$, $\frac{\delta+1}{2}$, differs by $2$ from the one appearing in $\eqref{rhs}$, $\frac{\delta-3}{2}$. Hence, in order to prove the equality \eqref{scd_der}, we need to exploit a differential equality relating ${}_{2}F_{1}(a,b,c,z)$ to ${}_{2}F_{1} (a,b',c,z)$, for any two parameters $b$ and $b'$ differing by an integer. Such a relation is provided by the following property:
\begin{lm}
\label{classical_prop}
\begin{equation}
\label{relation_hypergeom}
\frac{d}{dz} \left( z^{c-b}(1-z)^{a+b-c} {}_{2}F_{1}(a,b,c,z) \right) = (c-b) \, z^{c-b-1} (1-z)^{a+b-c-1} {}_{2}F_{1}(a,b-1,c,z).
\end{equation}
\end{lm}

\begin{proof}
Since the above relation does not seem easy to find in the litterature, we provide a proof. Note that the left-hand side of \eqref{relation_hypergeom} takes the form
\[ z^{c-b-1} (1-z)^{a+b-c-1} S(a,b,c,z),\]
where
\[\begin{split}
S(a,b,c,z) &= \sum_{k=0}^\infty\frac{(a)_{k} (b)_{k}}{k! (c)_{k}} \left[(k+c-b)(1-z)z^k - (a+b-c)z^{k+1}\right] \\
&= \sum_{k=0}^\infty \frac{(a)_{k} (b)_{k}}{k! (c)_{k}} \left[(k+c-b)z^k - (k+a)z^{k+1}\right].
\end{split}\]
Now, recalling that $(a)_{k} (k+a) = (a)_{k+1}$, it follows that 
\[\begin{split}
\sum_{k=0}^\infty \frac{(a)_{k} (b)_{k}}{k! (c)_{k}} \, (k+a)z^{k+1} &= \sum_{k=0}^\infty \frac{(a)_{k+1} (b)_{k}}{k! (c)_{k}} \, z^{k+1} \\
&=\sum_{k=1}^\infty \frac{(a)_k (b)_{k-1}}{(k-1)! \, (c)_{k-1}} \, z^k \\
&= \sum_{k=1}^\infty \frac{(a)_k (b)_{k-1}}{k! (c)_{k}} \, k(c+k-1) \, z^k.
\end{split}\]
Therefore, 
\[\begin{split} 
S(a,b,c,z) &= (c-b) + \sum_{k=1}^\infty \frac{(a)_k (b)_{k-1}}{k! (c)_{k}} \left[(b+k-1)(k+c-b) - k(c+k-1)\right] z^k.
\end{split}\]
Since, for all $k \geq 1$,  $(b+k-1)(k+c-b) - k(c+k-1) = (c-b)(b-1)$, and recalling that $(b-1) (b)_{k-1} = (b-1)_k$, we deduce that 
\[ S(a,b,c,z) = (c-b) + (c-b) \sum_{k=1}^\infty \frac{(a)_k (b-1)_{k}}{k! (c)_{k}} z^k = (c-b) \, {}_{2}F_{1}(a,b-1,c,z),\]
so the claim follows. 
\end{proof}
We exploit the relation provided by Lemma \ref{classical_prop} as follows. Let $ s \in (0,1)$, and $r \in (s,1)$. Setting $z := \frac{s(1-r)}{r(1-s)}$, we have
\[ 1-z = \frac{r-s}{r(1-s)}.\]
Therefore, equality \eqref{twopt} can be rewritten as follows
\[ E^{\delta}[X_{s}X_{r}] =  K(\delta) \, s(1-r) \, z^{-1/2} (1-z)^{\delta/2 +1} \, {}_{2}F_{1} \left( \frac{\delta+1}{2}, \frac{\delta+1}{2}, \frac{\delta}{2}, z \right) \]
where 
\[ K(\delta) := 2 \, \frac{\Gamma \left( \frac{\delta+1}{2} \right)^{2}}{ \Gamma \left( \frac{\delta}{2} \right)^{2}}. \]
Therefore, for all $r \in (s,1)$, we obtain, by the Leibniz formula and the chain rule
\begin{align*}
\frac{d}{dr} E^{\delta} \left[X_{r} X_{s}\right] &= - K(\delta) \, s z^{-1/2} (1-z)^{\delta/2 +1} \, {}_{2}F_{1} \left( \frac{\delta+1}{2}, \frac{\delta+1}{2}, \frac{\delta}{2}, z \right) \\
&+ K(\delta) \, s(1-r) \frac{dz}{dr} \frac{d}{dz} \left( z^{-1/2} (1-z)^{\delta/2 +1} \, {}_{2}F_{1} \left( \frac{\delta+1}{2}, \frac{\delta+1}{2}, \frac{\delta}{2}, z \right) \right). 
\end{align*}
But $\frac{dz}{dr} = - \frac{s}{r^{2}(1-s)}$, and, by Lemma \ref{classical_prop}, it holds
\[ \frac{d}{dz} \left( z^{-1/2} (1-z)^{\delta/2 +1} {}_{2}F_{1} \left( \frac{\delta+1}{2}, \frac{\delta+1}{2}, \frac{\delta}{2}, z \right) \right) = - \frac{1}{2} z^{-3/2} (1-z)^{\delta/2} {}_{2}F_{1} \left( \frac{\delta+1}{2}, \frac{\delta-1}{2}, \frac{\delta}{2}, z \right). \] 
Hence we obtain
\begin{align*}
\frac{d}{dr} E^{\delta} \left[X_{r} X_{s}\right] &= - K(\delta) \, s z^{-1/2} (1-z)^{\delta/2 +1} \, {}_{2}F_{1} \left( \frac{\delta+1}{2}, \frac{\delta+1}{2}, \frac{\delta}{2}, z \right) \\
&- K(\delta) \, s(1-r) \frac{s}{r^{2}(1-s)} \left(-\frac{1}{2} z^{-3/2} (1-z)^{\delta/2} {}_{2}F_{1} \left( \frac{\delta+1}{2}, \frac{\delta-1}{2}, \frac{\delta}{2}, z \right) \right) \\
&= - K(\delta) \, s z^{-1/2} (1-z)^{\delta/2 +1} {}_{2}F_{1} \left( \frac{\delta+1}{2}, \frac{\delta+1}{2}, \frac{\delta}{2}, z \right) \\
&+ K(\delta) \, \frac{1}{2} \frac{1-s}{1-r}  z^{1/2} (1-z)^{\delta/2} {}_{2}F_{1} \left( \frac{\delta+1}{2}, \frac{\delta-1}{2}, \frac{\delta}{2}, z \right). 
\end{align*}
Differentiating with respect to $r$ a second time, we obtain
\begin{align*}
\frac{d^{2}}{dr^{2}} E^{\delta} \left[X_{r} X_{s}\right] &= - K(\delta) \, s \frac{dz}{dr} \frac{d}{dz} \left\{ z^{-1/2} (1-z)^{\delta/2 +1} {}_{2}F_{1} \left( \frac{\delta+1}{2}, \frac{\delta+1}{2}, \frac{\delta}{2}, z \right) \right\} \\
&+ \frac{1}{2} K(\delta) \, \frac{1-s}{(1-r)^{2}}  z^{1/2} (1-z)^{\delta/2} {}_{2}F_{1} \left( \frac{\delta+1}{2}, \frac{\delta-1}{2}, \frac{\delta}{2}, z \right) \\ 
&+ \frac{1}{2} K(\delta) \, \frac{1-s}{1-r} \frac{dz}{dr} \frac{d}{dz} \left\{ z^{1/2} (1-z)^{\delta/2} {}_{2}F_{1} \left( \frac{\delta+1}{2}, \frac{\delta-1}{2}, \frac{\delta}{2}, z \right) \right\}. 
\end{align*}
Using again the expression for $\frac{dz}{dr}$, as well as Lemma \ref{classical_prop}, we deduce that
\begin{align*}
\frac{d^{2}}{dr^{2}} E^{\delta} \left[X_{r} X_{s}\right] &= K(\delta) \, s \frac{(1-r)}{r^{2}(1-s)} \left\{ -\frac{1}{2} z^{-3/2} (1-z)^{\delta/2} {}_{2}F_{1} \left( \frac{\delta+1}{2}, \frac{\delta-1}{2}, \frac{\delta}{2}, z \right) \right\} \\
&+ \frac{1}{2} K(\delta) \, \frac{1-s}{(1-r)^{2}}  z^{1/2} (1-z)^{\delta/2} {}_{2}F_{1} \left( \frac{\delta+1}{2}, \frac{\delta-1}{2}, \frac{\delta}{2}, z \right) \\ 
&- \frac{1}{2} K(\delta) \, \frac{1-s}{1-r} \frac{s}{r^{2}(1-s)} \left\{ \frac{1}{2} z^{-1/2} (1-z)^{\delta/2 -1} {}_{2}F_{1} \left( \frac{\delta+1}{2}, \frac{\delta-3}{2}, \frac{\delta}{2}, z \right) \right\}. 
\end{align*}
The first two terms cancel out, so that we obtain
\begin{align*}
\frac{d^{2}}{dr^{2}} E^{\delta} \left[X_{r} X_{s}\right] &= 
- \frac{K(\delta)}{4} \frac{s}{r^{2}(1-r)} z^{-1/2} (1-z)^{\delta/2 -1} {}_{2}F_{1} \left( \frac{\delta+1}{2}, \frac{\delta-3}{2}, \frac{\delta}{2}, z \right) \\
&= - \frac{K(\delta)}{4} \frac{(r-s)^{\delta/2-1} s^{1/2}}{ (1-r)^{3/2} r^{\frac{\delta+1}{2}} (1-s)^{\frac{\delta-3}{2}}} \ {}_{2}F_{1} \left( \frac{\delta+1}{2}, \frac{\delta-3}{2}, \frac{\delta}{2}, \frac{s(1-r)}{r(1-s)} \right)
\end{align*}
and, by \eqref{rhs}, the last expression is equal to
\[- \frac{\Gamma(\delta)}{4(\delta-2)} \langle \mu_{\delta-3} (\d b) , \Sigma^{\delta}_{r} (X_s | b) \rangle.\]
This yields the claim. 

\textbf{Second step:}

We now prove that equality \eqref{eq_distrib} holds. More precisely, for any test function $h \in C^{2}_{c}(0,1)$, we compute
\[ \int_{0}^{1} h''(r) \, E^{\delta} [X_{r} X_{s}] \, \d r. \]
Performing two successive integration by parts on the intervals $(0,s)$ and $(s,1)$, and recalling that $h$ has compact support in $(0,1)$ and is continuous at $s$, we obtain 
\begin{align}
\label{two_succ_ibps}
\int_{0}^{1} h''(r) E^{\delta} [X_{r} X_{s}] dr &= h(s) \left\{  \frac{d^{+}}{dr} E^{\delta} [X_{r} X_{s}]  - \frac{d^{-}}{dr} E^{\delta} [X_{r} X_{s}]  \right\} \\
& \nonumber + \int_{0}^{1} h(r) \frac{d^{2}}{dr^{2}} E^{\delta} [X_{r} X_{s}] \, \d r
\end{align}
where
\begin{equation}
\label{right_limit}
 \frac{d^{+}}{dr} E^{\delta} [X_{r} X_{s}] := \underset{r \searrow s}{\lim} \ \frac{d}{dr} E^{\delta} [X_{r} X_{s}] 
\end{equation}  
and 
\begin{equation}
\label{left_limit}
\frac{d^{-}}{dr} E^{\delta} [X_{r} X_{s}] := \underset{r \nearrow s}{\lim} \ \frac{d}{dr} E^{\delta} [X_{r} X_{s}]
\end{equation} 
are the right and left limits of the derivative of $E^{\delta} [X_{r} X_{s}]$ at $r=s$ (the existence of these limits will be justified herebelow).
By the first step, we readily know that the second term in the right-hand side above equals
\[- \frac{\Gamma(\delta)}{4(\delta-2)} \int_{0}^{1} \d r \, h(r) \, \langle \mu_{\delta-3} (\d b) , \Sigma^{\delta}_{r} (X_s | b) \rangle.\]
So there remains to establish the existence of and compute the limits \eqref{right_limit} and \eqref{left_limit}. For this, we use the following lemma:
\begin{lm}
\label{asymptotics}
Let $\alpha, \beta, \gamma \in \mathbb{C}$ such that $\gamma \notin \mathbb{Z}_{-}$, and $\gamma - \alpha - \beta \in \mathbb{R}^{*}_{-} \setminus \mathbb{Z}$. Then, for $z \in (0,1)$ tending  to $1$,
\begin{equation*}
{}_{2}F_{1}(\alpha, \beta, \gamma, z) \underset{z \to 1}{\sim} \frac{ \Gamma(\gamma) \Gamma(\alpha + \beta -\gamma)}{\Gamma(\alpha) \Gamma(\beta)} \, (1-z)^{\gamma-\alpha -\beta}.
\end{equation*}
\end{lm}
\begin{proof}
By Thm 8.5 in \cite{viola2016introduction}, the following equality holds for all $z \in (0,1)$:
\begin{align*}
& {}_{2}F_{1}(\alpha, \beta, \gamma, z) =  \frac{ \Gamma(\gamma) \Gamma(\gamma-\alpha - \beta)}{\Gamma(\gamma - \alpha) \Gamma(\gamma - \beta)} \, {}_{2}F_{1}(\alpha, \beta, \alpha + \beta - \gamma - 1, 1-z) \\
& + \frac{ \Gamma(\gamma) \Gamma(\alpha + \beta -\gamma)}{\Gamma(\alpha) \Gamma(\beta)} \, (1-z)^{\gamma-\alpha -\beta} \, {}_{2}F_{1}(\gamma-\alpha, \gamma -\beta, \gamma-\alpha - \beta +1, 1-z).
\end{align*}
Now, the functions ${}_{2}F_{1}(\alpha, \beta, \alpha + \beta - \gamma - 1, \cdot)$ and ${}_{2}F_{1}(\gamma-\alpha, \gamma -\beta, \gamma-\alpha - \beta +1, \cdot)$ are continuous at $0$ and take value $1$ there, while $(1-z)^{\gamma-\alpha -\beta} \to +\infty$ as $z \to 1$, since $\gamma-\alpha -\beta < 0$. The claim follows.
\end{proof}
Now, recalling the computations done in the first step, we have, for all  $r >s$,
\begin{align*}
\frac{d}{dr} E^{\delta} \left[X_{r} X_{s}\right] =&  - K(\delta) \, s z^{-1/2} (1-z)^{\delta/2 +1} {}_{2}F_{1} \left( \frac{\delta+1}{2}, \frac{\delta+1}{2}, \frac{\delta}{2}, z \right) \\
&+ K(\delta) \frac{1}{2} \frac{1-s}{1-r}  z^{1/2} (1-z)^{\delta/2} {}_{2}F_{1} \left( \frac{\delta+1}{2}, \frac{\delta-1}{2}, \frac{\delta}{2}, z \right) 
\end{align*}
where $z := \frac{s(1-r)}{r(1-s)} \in (0,1)$. Therefore, letting $r \searrow s$ and using Lemma \ref{asymptotics} we see that
\begin{align*}
\underset{r \searrow s}{\lim} \ \frac{d}{dr} E^{\delta} \left[X_{r} X_{s}\right]  =&- K(\delta) \frac{\Gamma \left( \frac{\delta}{2} \right) \Gamma(\frac{\delta}{2} + 1)}{\Gamma \left( \frac{\delta + 1}{2} \right)^{2}} s
+ \frac{1}{2} K(\delta) \frac{\Gamma\left( \frac{\delta}{2} \right)^{2}}{\Gamma \left( \frac{\delta + 1}{2} \right) \Gamma \left( \frac{\delta - 1}{2} \right)} \\  
=& -\delta s + \frac{\delta-1}{2}.
\end{align*}
Similarly, for all $r < s$, we have
\begin{align*}
\frac{d}{dr} E^{\delta} \left[X_{r} X_{s}\right] =&  K(\delta) \, (1-s) z^{-1/2} (1-z)^{\delta/2 +1} {}_{2}F_{1} \left( \frac{\delta+1}{2}, \frac{\delta+1}{2}, \frac{\delta}{2}, z \right) \\
&- \frac{1}{2} K(\delta) \, \frac{1}{2} \frac{s}{r}  z^{1/2} (1-z)^{\delta/2} {}_{2}F_{1} \left( \frac{\delta+1}{2}, \frac{\delta-1}{2}, \frac{\delta}{2}, z \right) 
\end{align*}
where $z := \frac{r(1-s)}{s(1-r)} \in (0,1)$. Therefore, letting $r \nearrow s$ and using Lemma \ref{asymptotics} we see that
\begin{align*}
\underset{r \nearrow s}{\lim} \ \frac{d}{dr} E^{\delta} \left[X_{r} X_{s}\right]  =& K(\delta) \frac{\Gamma \left( \frac{\delta}{2} \right) \Gamma(\frac{\delta}{2} + 1)}{\Gamma \left( \frac{\delta + 1}{2} \right)^{2}}  \, (1-s)
- \frac{1}{2} K(\delta) \frac{\Gamma\left( \frac{\delta}{2} \right)^{2}}{\Gamma \left( \frac{\delta + 1}{2} \right) \Gamma \left( \frac{\delta - 1}{2} \right)} \\  
=& \delta (1-s) - \frac{\delta-1}{2}.
\end{align*}
Therefore,  $\frac{d^{+}}{dr} E^{\delta} [X_{r} X_{s}]$ and  $\frac{d^{-}}{dr} E^{\delta} [X_{r} X_{s}]$ do indeed exist, and they satisfy
\begin{align*}
\frac{d^{+}}{dr} E^{\delta} \left[X_{r} X_{s}\right] - \frac{d^{-}}{dr} E^{\delta} \left[X_{r} X_{s}\right] =& \left(-\delta s + \frac{\delta-1}{2}\right) - \left(\delta (1-s) - \frac{\delta-1}{2}\right) \\
=& \ -1.
\end{align*}
Hence, \eqref{two_succ_ibps}, finally becomes
\begin{align*}
\label{two_succ_ibps}
\int_{0}^{1} h''(r) \, E^{\delta} [X_{r} X_{s}] \, \d r &= - h(s) \\
& - \frac{\Gamma(\delta)}{4(\delta-2)} \int_{0}^{1} \d r \, h(r) \, \langle \mu_{\delta-3} (\d b) , \Sigma^{\delta}_{r} (X_s | b) \rangle,
\end{align*}
which concludes the proof of Theorem \ref{ibpf_linear_test_functn}.
\end{proof}

\section{A more general class of functionals}

More generally, given a continuous function $\varphi : [0,1] \to \mathbb{R}$ and a finite Borel measure $m$ on $[0,1]$, we can consider the functional $\Phi$ defined on $C([0,1])$ by
\begin{equation}
\label{linear_times_exp_func}
 \Phi(X) := \langle \varphi , X \rangle \exp \left( -\langle m, X^{2} \rangle \right), \quad X \in C([0,1]),
\end{equation} 
which is a product of functionals of the form \eqref{linear_functional} and \eqref{exp_square_functional}.  
Note that, as soon as $\varphi \neq0$ and $m \neq 0$, $\Phi$ is neither of of the form \eqref{exp_square_functional} nor of the form \eqref{linear_functional}, and cannot be written as a linear combination of such functionals.
However, using the same arguments as above, and interpreting $\exp \left( -\langle m, X^{2} \rangle \right) P^\delta({\rm{d}} X)$ as the law (up to a constant) of a time-changed  Bessel bridge (see \cite[Lemma 3.3]{EladAltman2019}), one can show that the IbPF above also hold for a functional $\Phi$ of the form \eqref{linear_times_exp_func}. Since the techniques are the same as those presented above, but the computations much lenghtier, we do not provide a proof of this fact.

\bibliographystyle{amsplain}
\bibliography{Hypergeometric_Functions_Revised}

\end{document}